\documentclass[12pt]{amsart}
\usepackage{amsmath}
\usepackage{amsfonts}
\usepackage{amsthm}
\usepackage{amssymb}
\usepackage{amscd}
\usepackage[all]{xy}
\usepackage{enumerate}



\theoremstyle{plain}
\newtheorem{thm}{Theorem}[subsection]

\newtheorem{prop}[thm]{Proposition}

\newtheorem{cor}[thm]{Corollary}

\newtheorem{lem}[thm]{Lemma}

\theoremstyle{definition}
\newtheorem{defn}[thm]{Definition}

\newtheorem{rmk}[thm]{Remark}
\newcommand{\sA}{\mathcal{A}}
\newcommand{\sB}{\mathcal{B}}
\newcommand{\sC}{\mathcal{C}}

\newcommand{\sE}{\mathcal{E}}
\newcommand{\sF}{\mathcal{F}}
\newcommand{\sG}{\mathcal{G}}

\newcommand{\sI}{\mathcal{I}}

\newcommand{\sK}{\mathcal{K}}
\newcommand{\sL}{\mathcal{L}}
\newcommand{\sN}{\mathcal{N}}
\newcommand{\sM}{\mathcal{M}}
\newcommand{\sO}{\mathcal{O}}

\newcommand{\mC}{\mathbb{C}}

\newcommand{\mP}{\mathbb{P}}

\newcommand{\Ima}{\mathrm{Im}\,}

\newcommand{\Hom}{\mathrm{Hom}}

\numberwithin{equation}{section}

\newcommand{\beba}  {\begin{equation}\begin{array}{rcl}}

\newcommand{\eaee}  {\end{array}\end{equation}}

\let\oldtocsection=\tocsection

\let\oldtocsubsection=\tocsubsection

\let\oldtocsubsubsection=\tocsubsubsection

\renewcommand{\tocsection}[2]{\hspace{0em}\oldtocsection{#1}{#2}}
\renewcommand{\tocsubsection}[2]{\hspace{1em}\oldtocsubsection{#1}{#2}}
\renewcommand{\tocsubsubsection}[2]{\hspace{2em}\oldtocsubsubsection{#1}{#2}}

\title{A note on Torelli-type theorems for Gorenstein curves}

\author{Luca Rizzi}
\address{D.I.M.I. \\
the University of Udine\\
Udine, 33100 Italy\\
\texttt{rizzi.luca@spes.uniud.it}}

\author{Francesco Zucconi}

\address{D.I.M.I. \\
the University of Udine\\
Udine, 33100 Italy\\
\texttt{Francesco.Zucconi@dimi.uniud.it}}

\begin{document}

\markboth{Rizzi and Zucconi}{A note on Torelli-type theorems for Gorenstein curves}

\begin{abstract}
Using the notion of generalized divisors introduced by Hartshorne (see \cite{H2}), we adapt the theory of adjoint forms to the case of Gorenstein curves. We show an infinitesimal Torelli-type theorem for vector bundles on Gorenstein curves. We also construct explicit counterexamples to the infinitesimal Torelli claim in the case of a reduced reducible Gorenstein curve.
\end{abstract}

\maketitle
\tableofcontents

\section{Introduction} 

Let
$\pi\colon \sC\to \Delta$ be a smooth family of smooth curves of genus $g$ over a complex polydisk $\Delta$ and let  $0$ be a point of $\Delta$. 
By Ehresmann's theorem, after possibly shrinking the base,  we have an isomorphism between 
$H^{1}(C_{t},\mathbb C)$ and $V= H^{1}(C_{{0}},\mathbb C)$ where
$t\in \Delta$ and $C_t$ is the fiber over $t$. It remains defined the {\it{period map}} 
$$\Phi:\Delta\rightarrow \mathbb G={\rm{Grass}}(g, V)$$
({\it cf.}~e.g., \cite{gr-s}) which is the map that to each $t\in \Delta$ associates the subspace
$H^{10}(C_{t})$ of $V$ where $g={\rm{dim}}_{\mathbb C}H^{10}(C_{{0}})$. In a much more general context Griffiths in \cite{grif1} and \cite{grif2} 
proved that $\Phi$ is holomorphic, that the image of the differential
$d\Phi_{0}:T_{\Delta,{0}}\rightarrow T_{{\mathbb{G}}, [H^{10}(C_0)]}$ 
is actually contained in ${\rm{Hom}}(H^{10}(C_{0},\mathbb C), 
H^{0,1}(C_{0},\mathbb
C))$
and finally 
he showed that $d\Phi_{0}$ is
the composition of the Kodaira-Spencer map $ T_{\Delta,0}\rightarrow
H^{1}(C_{0},\Theta_{C_{0}})$ with the map given by the cup product and the
interior product:
\begin{equation}
\cup\colon H^{1}(C_{0},\Theta_{C_{0}})\rightarrow  
\Hom (H^{0}(C_0, \Omega^{1}_{C_0}), H^{1}(C_0, \sO_{C_0}))
\end{equation}
where $\Theta_{X}$ is the tangent sheaf of a variety $X$ and where we use the isomorphisms $H^{10}(C_0)\sim H^{0} (C_0, \Omega^{1}_{C_0})$ and 
$H^{01}(C_0)\sim H^{1}(C_0, \sO_{C_0})$. If $\pi \colon \sC \to \Delta$ is the local Kuranishi family ({\it cf.}~\cite{K-M}, \cite{K}) of $C_0$ the Kodaira-Spencer map $ T_{\Delta,0}\rightarrow
H^{1}(C_{0},\Theta_{C_{0}})$ is an isomorphism and the infinitesimal Torelli problem asks if the period map $\pi \colon \sC \to \Delta$ ({\it cf.}~\cite{K-M}, \cite{K}) of $C_0 $ is locally an immersion. In other words we say that the infinitesimal 
Torelli theorem holds for $C_0$ if $\cup\colon H^{1}(C_{0},\Theta_{C_{0}})\rightarrow  
\Hom (H^{0}(C_0, \Omega^{1}_{C_0}), H^{1}(C_0, \sO_{C_0}))$ is injective. 

For a smooth algebraic curve $C$ of genus $g\geq 1$ the infinitesimal Torelli theorem
holds iff $g(C) = 1,2$ or iff
$g(C) \geq 3$ and $C$ is not hyperelliptic ({\it cf.}~\cite{to}, \cite{andr}, 
\cite{we}, \cite{OS}). Actually it holds a more deep result known as the Torelli theorem. Call $\sM_g$ the coarse moduli space of complete non-singular curves of genus $g$ and $\sA_g$ the moduli space of principally polarized abelian varieties of dimension $g$. Then the Torelli map $\tau_g\colon \sM_g\to\sA_g$ which sends the isomorphism class of a curve to the isomorphism class of its Jacobian is injective. There is a large literature concerning the problem of extending the Torelli map to a morphism $\overline{\tau_g}\colon \overline{\sM_g}\to \overline{\sA_g}$ where $\overline{\sM_g}$ is the Deligne-Mumford compactification of $\sM_g$ (see \cite{Na} and \cite{CV}) and $\overline{\sA_g}$  is a suitable compactification of ${\sA_g}$. Clearly the problem varies according to the chosen compactification of $\sA_g$: see \cite{CV}.
 
In the case of irreducible stable  curves Namikawa proved that the canonical map from the open set of irreducible stable curves to the normalization of the Satake compactification of $\sA_g$ is 
injective; see: \cite[Theorem 7 page 245]{Na}. He proved that if $C$ is an irreducible stable curve of genus $g$ whose normalization is a non-hyperelliptic curve of genus $>2$ then $C$ is uniquely determined by its generalized Jacobian \cite[Proposition 9 page 245]{Na}. He also showed that the above map can't be injective over the divisor $\sN=\cup_{i=1}^{[\frac{g}{2}]}\sN_i$ where $\sN_i$ is the divisor whose general points correspond to stable curves with two non-singular irreducible components $C_1$, $C_2$ with genus $i$ and $g-i$ meeting at one point.

Note that the homomorphism
\begin{equation}
\label{problema}
\cup\colon \text{Ext}^1(\Omega^1_C,\sO_C)\rightarrow
\Hom (H^{0}(C_0, \Omega^{1}_{C}), H^{1}(C, \sO_{C}))
\end{equation} fails to be injective if there exists $\xi\in {\rm{Ext}}^1(\Omega^1_C,O_C)$, $\xi\neq 0$ such that for one of the associated extensions
\begin{equation}
0\to \sO_C\to \sE\to \Omega^1_C\to 0
\end{equation} the coboundary homomorphism $\partial_{\xi}\colon H^0(C,\Omega^1_C)\to H^1(C,\sO_C)$ is trivial. 

In this paper we study the linear part of the infinitesimal Torelli problem, that is the injectivity of (\ref{problema}). In particular we focus on the general case of Gorenstein curves.

We study the following problem: let $C$ be a Gorenstein curve and let $\xi$ be an element of ${\rm{Ext}}^1(\sL,O_C)$ such that for one of the associated extensions
\begin{equation}
0\to \sO_C\to \sE\to \sL\to 0
\end{equation} it holds that the coboundary homomorphism $\partial_{\xi}\colon H^0(C,\Omega^1_C)\to H^1(C,\sO_C)$ is trivial. Is it true that $\xi=0$? We prove:
\begin{thm}
Let $C$ be an irreducible Gorenstein curve, and let $\sL$ be a locally free sheaf of rank one on $C$. Consider the extension
\begin{equation}
0\to\sO_C\to\sE\to\sL\to0
\end{equation} given by an element $\xi\in {\rm{Ext}}^1(\sL,\sO_C)\cong H^1(C,\sL^\vee)$. Call $F$ the fixed part of the linear system associated to $\sL$ and assume that $F$ does not contain singularities; call $M$ its mobile part. Assume that the map $\phi_M$ given by $M$ is of degree one and that $l:=\dim |M|\geq3$. If the cohomology map
\begin{equation}
H^0(C,\sE)\to H^0(C,\sL)
\end{equation} is surjective, then $\xi\in \ker(H^1(C,\sF^\vee)\to H^1(C,\sF^\vee(F)))$.
\end{thm}
For the proof see Theorem \ref{teorema1}. We stress that we use the notion of generalized divisor introduced in \cite{H2}. We also construct a theory of adjoint forms which extends the previous one of \cite[Section 1]{CP} and \cite[Section 3]{PZ} to the case of a Gorenstein irreducible curve following the approach of \cite[Theorem 2.1.7]{RZ1}. For this reason we need a control on the fixed part $F$.

As a Corollary we obtain 

\begin{thm}
\label{corint}
Let $C$ be an irreducible Gorenstein curve of genus $2$ or non-hyperelliptic of genus $\geq3$, $\xi\in H^1(C,\omega_C^\vee)$ such that $\partial_{\xi}=0$. Then $\xi=0$. 
\end{thm}  We recall that for an irreducible Gorenstein curve $C$ in general $\text{Ext}^1(\Omega^1_C,\sO_C)$ is different from $\text{Ext}^1(\omega_C,\sO_C)$ and Theorem \ref{corint} is basically the infinitesimal Torelli theorem for the deformations of Gorenstein curves which do not smooth the singularities. 

In the final part of the paper we give an explicit proof that in the case of reducible curves a statement as the one of Theorem \ref{teorema1} can not be true.

\subsection{Acknowledgment} 
This research is supported by MIUR funds, 
PRIN project {\it Geometria delle variet\`a algebriche} (2010), coordinator A. Verra.

The authors would like to thank Edoardo Ballico, Miguel \'Angel Barja, Luca Cesarano and Gian Pietro Pirola for very useful conversations on this topic.

%This theorem gives some information on the deformations of a Gorenstein curve; see Theorem \ref{deformazionigorenstein}.
%\begin{thm}
%Take an infinitesimal deformation  $\xi\in\text{Ext}^1(\Omega^1_C,\sO_C)$ which does not smooth the nodes of $C$. Then $\xi$ can be lifted to an element $\tilde{\xi}$ of $\text{Ext}^1(\omega_C,\sO_C)$. Assume that $C$ is of genus $2$ or non-hyperelliptic of genus $\geq3$. If $\partial_{\tilde{\xi}}=0$ then $\xi=0$.
%\end{thm}

\section{Gorenstein curves}
\label{sezione1}

A \emph{Gorenstein} curve is a projective curve $C$ such that the dualizing sheaf $\omega_C$ is invertible.
Given a Gorenstein curve $C$, we can define the sheaf of K\"ahler differentials $\Omega^1_C$; see for instance \cite[Chapter II Section 8]{H1}. This sheaf is not locally free as in the smooth case and in general it has a torsion part.
On the other hand the \emph{dualizing} sheaf $\omega_C$ is, by definition of Gorenstein curve, a locally free invertible sheaf on the curve $C$.

The sheaves $\Omega^1_C$ and $\omega_C$ are isomorphic outside the singular locus of $C$ and this isomorphism can be completed to a morphism $\rho\colon \Omega^1_C\to \omega_C$, see \cite[Page 244]{BG}.
The morphism $\rho$ fits into the following exact sequence which will be crucial in this paper:

\begin{equation}
\label{sequenzalunga}
0\to K\to \Omega^1_C\stackrel{\rho}{\rightarrow}\omega_C\to N\to 0.
\end{equation} The kernel and the cokernel of $\rho$, denoted by $K$ and $N$ respectively, are torsion sheaves supported on the singularities.

The group ${\rm{Ext}}^1(\Omega^1_C,O_C)$ parametrizes the infinitesimal deformations of $C$ {\it cf.} \cite[Corollary 1.1.11]{S}.
The following result gives some important information on this group:
\begin{prop}
\label{proposizione1}
We have the following exact sequence 
\begin{equation}
\label{localtoglobal1}
0\to H^1(C,\mathcal{H}\textit{om}(\Omega_C^1,\sO_C))\to \text{{\rm Ext}}^1(\Omega^1_C,\sO_C)\stackrel{\mu}{\to} R \to0
\end{equation} where $R$ is a sheaf supported on the singularities of $C$.
\end{prop}
\begin{proof}
This sequence can be obtained by the five term sequence associated to the local to global spectral sequence of $\text{Ext}$'s. Here we give an alternative proof because we will use some of its steps throughout the paper.

Sequence (\ref{sequenzalunga}) splits into the following short exact sequences:
\begin{equation}
0\to K\to \Omega^1_C\to \hat{\omega}\to0
\label{seq1}
\end{equation} and
\begin{equation}
0\to \hat{\omega}\to \omega_C\to N\to0.
\label{seq2}
\end{equation}
Dualizing (\ref{seq1}) we have 
{\footnotesize
\begin{equation}
\label{lunga1}
0\to \mathcal{H}\textit{om}(\hat{\omega},\mathcal{O}_C)\to \mathcal{H}\textit{om}(\Omega^1_C,\sO_C)\to0\to \mathcal{E}\textit{xt}^1(\hat{\omega},\mathcal{O}_C)\to \mathcal{E}\textit{xt}^1(\Omega^1_C,\sO_C)\to \mathcal{E}\textit{xt}^1(K,\sO_C)\to0
\end{equation}} and in particular we deduce that the dual of the sheaf $\Omega_C^1$ is isomorphic to the dual of the sheaf $\hat{\omega}$.
Taking the dual of (\ref{seq2}) we obtain the long exact sequence 
{\footnotesize
\begin{equation}
0\to \mathcal{H}\textit{om}(\omega_C,\mathcal{O}_C)\to \mathcal{H}\textit{om}(\hat{\omega},\sO_C)\to \mathcal{E}\textit{xt}^1(N,\mathcal{O}_C)\to \mathcal{E}\textit{xt}^1(\omega_C,\sO_C)\to \mathcal{E}\textit{xt}^1(\hat{\omega},\sO_C)\to0.
\end{equation}} Since the dualizing sheaf is locally free we have by \cite[Proposition 6.7]{H1} that $\mathcal{E}\textit{xt}^1(\omega_C,\sO_C)=0$, and we deduce that $\mathcal{E}\textit{xt}^1(\hat{\omega},\sO_C)=0$; see also \cite[Lemma 1.1]{H2}.

Now we apply the functor $\text{Hom}(-,\sO_C)$ to (\ref{seq1}) and we have
{\footnotesize
\begin{equation}
0\to \text{Hom}(\hat{\omega},\mathcal{O}_C)\to \text{Hom}(\Omega^1_C,\sO_C)\to0\to \text{Ext}^1(\hat{\omega},\mathcal{O}_C)\to \text{Ext}^1(\Omega^1_C,\sO_C)\to \text{Ext}^1(K,\sO_C)\to0.
\end{equation}} Then
\begin{equation}
\label{localtoglobal}
0\to \text{Ext}^1(\hat{\omega},\mathcal{O}_C)\to \text{Ext}^1(\Omega^1_C,\sO_C)\to \text{Ext}^1(K,\sO_C)\to0
\end{equation} is exact. 
The fact that the sheaf $K$ is supported on the singularities implies that also $R:=\text{Ext}^1(K,\sO_C)$ is a torsion sheaf. On the other hand, the kernel $\text{Ext}^1(\hat{\omega},\mathcal{O}_C)$ parametrizes the extensions of the sheaf $\hat{\omega}$ by the structure sheaf $\sO_C$, that is isomorphism classes of exact sequences of the form
\begin{equation}
0\to \sO_C\to \mathcal{E}\to \hat{\omega}\to0.
\end{equation} By (\ref{lunga1}) and the fact that the sheaves are reflexive (see \cite[Lemma 1.1]{H2}) we have that
%The dual of this sequence is 
%\begin{equation}
%0\to\mathcal{H}\textit{om}(\hat{\omega},\sO_C)\to\mathcal{H}\textit{om}(\sE,\sO_C)\to \sO_C\to0
%\end{equation} which is again short exact since $\mathcal{E}\textit{xt}^1(\hat{\omega},\sO_C)=0$. By (\ref{lunga1}) this is isomorphic to
%\begin{equation}
%0\to\mathcal{H}\textit{om}(\Omega_C^1,\sO_C)\to\mathcal{H}\textit{om}(\sE,\sO_C)\to \sO_C\to0.
%\end{equation} The image of $1\in H^0(C,\sO_C)$ via the connecting morphism 
%\begin{equation}
%H^0(C,\sO_C)\to H^1(C,\mathcal{H}\textit{om}(\Omega_C^1,\sO_C))
%\end{equation} characterizes the sequence, hence we deduce the isomorphism 
\begin{equation}
\label{isomorfismo}
H^1(C,\mathcal{H}\textit{om}(\Omega_C^1,\sO_C))\cong \text{Ext}^1(\hat{\omega},\mathcal{O}_C).
\end{equation}

Sequence (\ref{localtoglobal}) is then
\begin{equation}
0\to H^1(C,\mathcal{H}\textit{om}(\Omega_C^1,\sO_C))\to \text{Ext}^1(\Omega^1_C,\sO_C)\to R\to0.
\end{equation} 
\end{proof}

\subsection{The nodal case} Nodal curves are an interesting example of Gorenstein curves. 
Recall that a point in a projective curve is a \emph{node} if it has a neighborhood in the analytic topology which is isomorphic to a neighborhood of the origin in the space $(xy=0)\subset\mC^2$. A \emph{nodal curve} is a curve with only nodes as singularities. We denote by $\nu\colon \widetilde{C}\to C$ its normalization.
In this case the situation described above is more explicit. 

Around a node $P$ given locally by $xy=0$, the sheaf $\Omega^1_C$ is generated by $dx$ and $dy$ with the relation $ydx+xdy=0$; see \cite[Chapter X]{ACG}. 

On the other hand the dualizing sheaf is defined as follows. Consider $P_1,\ldots,P_n$ the nodes of the curve and $Q_1,Q'_1,\ldots,Q_n,Q'_n$ their preimages in the normalization $\widetilde{C}$ of $C$. Then $\omega_C$ is the subsheaf of
\begin{equation}
\nu_*(\omega_{\widetilde{C}}(\sum{Q_i+Q'_i}))
\end{equation} given by the sections $\sigma$ with opposite residues in $Q_i$ and $Q_i'$, that is
\begin{equation}
\text{Res}_{Q_i}(\sigma)+\text{Res}_{Q_i'}(\sigma)=0.
\end{equation} A local generator for $\omega_C$ in a neighborhood of a node is by adjunction $\frac{dx\wedge dy}{F}$, where $F$ is a local equation for the curve.

Locally near a node $P$, $\rho$ is given by 
\begin{equation}
\rho(dx)=x\frac{dx\wedge dy}{F}
\end{equation}and
\begin{equation}
\rho(dy)=y\frac{dx\wedge dy}{F}.
\end{equation}

 The stalk $K_P$ of the kernel of (\ref{sequenzalunga}) is the $\mathbb{C}$-vector space generated by $xdy=-ydx$. To understand the cokernel $N_P$, we note that the image of $\rho$, denoted by $\hat{\omega}$, is generated by the ideal $(x,y)$ in $\omega$, that is $\hat{\omega}=(x,y)\cdot\omega$. $N_P$ is then $\omega/(x,y)\cdot\omega$.

By the explicit description of $K$ in the nodal case it holds that $R=\bigoplus_i\mathbb{C}_{P_i} $ and (\ref{localtoglobal1}) is 
\begin{equation}
\label{localtoglobal2}
0\to H^1(C,\mathcal{H}\textit{om}(\Omega_C^1,\sO_C))\to \text{Ext}^1(\Omega^1_C,\sO_C)\to\bigoplus_i\mathbb{C}_{P_i} \to0;
\end{equation}see for instance \cite[Chapter XI]{ACG}.

The meaning of this sequence is that there are two kinds of infinitesimal deformations of $C$. The deformations corresponding to the cokernel $\bigoplus_i\mathbb{C}_{P_i}$ are those that give the smoothing of the nodes. More precisely the generator of $\mathbb{C}_{P_i}$ corresponds to the infinitesimal deformation given by $xy=\epsilon$ around the node $P_i$ glued together with the trivial deformation outside a neighborhood of $P_i$.

On the other hand the deformations coming from the kernel $H^1(C,\mathcal{H}\textit{om}(\Omega_C^1,\sO_C))$ are locally trivial around the nodes. They can also be seen as deformations of the pointed curve $(\widetilde{C}, Q_1,Q_1',\ldots,Q_n,Q_n')$. This comes from the fact that 
\begin{equation}
\label{isonorm}
\mathcal{H}\textit{om}(\Omega_C^1,\sO_C)\cong\nu_*(\Theta_{\widetilde{C}}(-\sum (Q_i+Q_i'))).
\end{equation}

Consider $\xi \in \text{Ext}^1(\Omega^1_C,\sO_C)$ corresponding to the exact sequence
\begin{equation}
0\to\sO_C\to \Omega^1_{S}|_C\to \Omega^1_C\to0.
\end{equation} If $\xi$ is a generator in the cokernel of (\ref{localtoglobal2}), then we note that the sheaf $\Omega^1_{S}|_C$ is locally free in a neighborhood of the node which is smoothed by the deformation. More precisely this sheaf is generated by the sections $dx$ and $dy$ with no relations. The section $1$ of the structure sheaf $\sO_C$ is sent locally to the $1$-form $xdy+ydx$ of $\Omega^1_{S}|_C$, so that in $\Omega^1_C$ we have the relation $xdy+ydx=0$.

We want now to give a Torelli-type theorem for the deformations coming from the kernel of (\ref{localtoglobal1}). For this purpose we introduce the adjoint theory in the case of a Gorenstein curve.

\section{Adjoint theory for irreducible Gorenstein curves}

Let $\sF$ and $\sL$ be two locally free sheaves of rank one on $C$. Consider the exact sequence of locally free sheaves
\begin{equation}
\label{sequenza}
0\to\sL\to \sE\to\sF\to 0
\end{equation} associated to an element $\xi\in \text{Ext}^1(\sF,\sL)$. Since $\sL$ is locally free of rank one we can tensor this sequence by $\sL^\vee$ and take the dual. Hence we obtain the exact sequence
\begin{equation}
0\to\sL\otimes\sF^\vee\to\sL\otimes \sE^\vee\to\sO_C\to 0.
\end{equation} As in the smooth case the image of $1\in H^0(C,\sO_C)$ via the morphism 
\begin{equation}
H^0(C,\sO_C)\to H^1(C,\sL\otimes\sF^\vee)
\end{equation} characterizes the extension and gives an isomorphism $\text{Ext}^1(\sF,\sL)\cong H^1(C,\sL\otimes\sF^\vee)$.

Let $\partial_\xi \colon H^0(C,\sF)\to H^1(C,\sL)$ be the connecting homomorphism of (\ref{sequenza}), and let $W\subset \ker(\partial_\xi)$ be a vector subspace of dimension $2$. Choose a basis $\mathcal{B}:=\{\eta_1,\eta_{2}\}$ of $W$. By definition we can take liftings $s_1,s_{2}\in H^0(C,\sE)$ of the sections $\eta_1,\eta_{2}$. 

Consider the base locus of $W$. It consists of a finite number of smooth points and singularities. Denote by $D_W$ the Cartier divisor associated to the smooth points in the base locus, and by $\sI_W$ the ideal of the singular points contained in the base locus. The ideal $\sI_W$ is an \emph{effective generalized divisor} on the curve $C$. 

The theory of generalized divisor on Gorenstein curves can be found in \cite{H2}. Here we briefly recall that a \emph{generalized divisor} on $C$ is a nonzero subsheaf of the constant sheaf of the function field $\sK$ which is also a coherent $\sO_C$-module. A generalized divisor is \emph{effective} if it is a nonzero ideal of $\sO_C$, that is if it corresponds to a $0$-dimensional closed subscheme of $C$. The inverse of the generalized divisor $\sI$ is locally given by $\sI^{-1}:=\{f\in\sK\mid f\cdot\sI\subset\sO_C\}$.

The sections $\eta_1$ and $\eta_2$ generate the sheaf $\sF(-D_W)\otimes \sI_W$ and by local computation we have the following short exact sequence
\begin{equation}
\label{sequenzaker}
0\to\sF^\vee(D_W)\otimes \sI_W^{-1}\to W\otimes\sO_C\to \sF(-D_W)\otimes \sI_W\to0
\end{equation} which fits into the  diagram
\begin{equation}
\label{diagramma0}
\xymatrix {0 \ar[r] &\sF^\vee(D_W)\otimes \sI_W^{-1} \ar[r] &W\otimes\sO_C \ar[r]\ar[d]^{(s_1,s_2)} &  \sF(-D_W)\otimes \sI_W\ar[r]\ar[d]&0\\
0 \ar[r] & \sL \ar[r] & \sE\ar[r] & \mathcal{F} \ar[r]&0.
}
\end{equation} We can complete the diagram with a morphism $\omega\colon \sF^\vee(D_W)\otimes \sI_W^{-1}\to\sL$, that is $$\omega\in \text{Hom}(\sF^\vee(D_W)\otimes \sI_W^{-1},\sL).$$ The morphism $\omega$ depends on the choice of the liftings $s_1, s_2$.

\begin{defn}
The section $\omega\in \text{Hom}(\sF^\vee(D_W)\otimes \sI_W^{-1},\sL)$ is called an adjoint of $W$ and $\xi$.
\end{defn}

We want to study the condition $\omega\in \Ima \Phi_{\sB}$ where
\begin{equation}
\label{aggiuntazero0}
 \Phi_{\sB}\colon \text{Hom}(W\otimes \sO_C,\sL)\to \text{Hom}(\sF^\vee(D_W)\otimes \sI_W^{-1},\sL).
\end{equation}
\begin{rmk}
The adjoint $\omega$ depends on the choice of the liftings $s_1, s_2$, whereas the above condition does not.
To be more precise if we change liftings $s'_1, s'_2$ and construct the corresponding adjoint $\omega'$. We have that $\omega\neq\omega'$ in general, but $\omega-\omega' \in\Ima \Phi_{\sB}$.
 %if and only if $\omega' \in \Ima(\text{Hom}(W\otimes \sO_C,\sL)\to \text{Hom}(\sF^\vee(D_W)\otimes \sI_W^{-1},\sL))$.
\end{rmk}
\begin{rmk}
Consider another basis $\mathcal{B}':=\{\eta_1',\eta_{2}'\}$ of $W$ and let $A$ be the matrix of the basis change. 
The sections $s_1',s_{2}'$ obtained from $s_1,s_{2}$ through the matrix $A$ are liftings of $\eta_1',\eta_{2}'$. It is easy to see composing with $A$ that $\omega \in \Ima \Phi_{\sB}$ if and only if $\omega' \in \Ima \Phi_{\sB'}$ where $\omega'$ is the adjoint constructed from $\eta_1',\eta_{2}'$.
\end{rmk}

\begin{thm}
\label{aggiunta0}
Let $C$ be an irreducible Gorenstein curve with $n$ singular points. Let $\sF$, $\sL$ be invertible sheaves on $C$. Consider $\xi\in \text{{\rm Ext}}^1(\sF,\sL)$ associated to the extension (\ref{sequenza}). Define $W=\left\langle \eta_1,\eta_{2}\right\rangle\subset\ker({\partial_\xi})\subset H^0(C,\mathcal{F})$ and $\omega$ as above. Call $\xi_{D_W}$ the image of $\xi$ via the morphism
\begin{equation}
\text{{\rm Ext}}^1(\sF,\sL)\stackrel{\rho}{\rightarrow} \text{{\rm Ext}}^1(\sF(-D_W)\otimes\sI_W,\sL).
\end{equation}
We have that $\omega\in\Ima(\text{{\rm Hom}}(W\otimes \sO_C,\sL)\to \text{{\rm Hom}}(\sF^\vee(D_W)\otimes \sI_W^{-1},\sL))$ if and only if $\xi_{D_W}=0$.

\end{thm}

\begin{proof} Set $\sB:=\{\eta_1,\eta_2\}$.
Take diagram (\ref{diagramma0}) and apply the functor $\text{Hom}(-,\sL)$. We obtain
{%\footnotesize
\begin{equation}
\label{diagramma10}
\xymatrix {\text{Hom}(\sE,\sL) \ar[r]\ar[d]& \text{Hom}(\sL,\sL) \ar[r]\ar[d]^{\beta}& \text{Ext}^1(\sF,\sL)\ar[d]^\rho\\
 \text{Hom}(W\otimes \sO_C,\sL)\ar[r]^-{\Phi_{\sB}} & \text{Hom}(\sF^\vee(D_W)\otimes \sI_W^{-1},\sL) \ar[r]^-{\delta}&\text{Ext}^1(\sF(-D_W)\otimes\sI_W,\sL).
}
\end{equation}}

Obviously $\beta(id)=\omega$ and, by commutativity, $\delta(\beta(id))=\xi_{D_W}$. We have then $\xi_{D_W}=0$ if and only if $\omega\in  \Ima \Phi_{\sB}$.
\end{proof}

\subsection{Non singular base locus}
If the base locus of $\eta_1$ and $\eta_2$ does not contain singular points, then the situation is easier and it can be more explicitly described as follows.

Consider again the liftings $s_1,s_2\in H^0(C,\sE)$.
The natural map
\begin{equation*}
\Lambda^2\colon\bigwedge^{2}H^0(C,\sE)\to H^0(C,\bigwedge^2\sE)
\end{equation*} defines the section
\begin{equation}
\label{omega}
\omega:=\Lambda^2(s_1\wedge s_{2}).
\end{equation}

It is easy to see that $\omega$ is in the image of the natural injection $\det\sE(-D_W)\to \det\sE$.

Call $\tilde{\eta}_i	\in H^0(C,\sF(-D_W))$ the sections corresponding to the $\eta_i$'s via $H^0(C,\sF(-D_W))\to H^0(C,\sF)$.
Since our sheaves are locally free and the base locus does not contain the singular points, sequence (\ref{sequenzaker}) is now given by
\begin{equation}
\label{castelnuovo}
0\to\mathcal{F}^\vee(D_W)\stackrel{i}{\rightarrow}\mathcal{O}_C\oplus\mathcal{O}_C\stackrel{\nu}{\rightarrow}\mathcal{F}(-D_W)\to 0.
\end{equation} This is basically the well known Castelnuovo's base point free pencil trick and the morphism $i$ is given by the contraction with $-\tilde{\eta}_1$ and $\tilde{\eta}_2$, while $\nu$ is given by the evaluation with $\tilde{\eta}_2$ on the first component and $\tilde{\eta}_1$ on the second one.

%\begin{proof}
%Everything is clear since $\sF$ is locally free of rank one and the sections $\tilde{\eta}_1,\tilde{\eta}_2$ generate the sheaf $\sF(-D_W)$. 
%\end{proof} This sequence is the analog of (\ref{sequenzaker}) in this case.

It is easy to see by local computation that sequence (\ref{castelnuovo}) fits into the following commutative diagram
\begin{equation}
\label{diagramma}
\xymatrix {0 \ar[r] &\sF^\vee(D_W) \ar[r]\ar[d]^-{\omega} &W\otimes\sO_C \ar[r]\ar[d]^{(s_1,s_2)} &  \sF(-D_W)\ar[r]\ar[d]&0\\
0 \ar[r] & \sL \ar[r] & \sE\ar[r] & \mathcal{F} \ar[r]&0.
}
\end{equation}
The morphism $W\otimes\sO_C\to\sE$ is given by the contraction with the sections $s_1$ and $s_2$, the morphism $\mathcal{F}^\vee(D_W)\to\sL$ by the contraction with the adjoint $\omega$.
Therefore condition (\ref{aggiuntazero0}) can be written as
\begin{equation}
\label{aggiuntazero1}
\omega \in \Ima(H^0(C,\sL)\otimes \left\langle s_1,s_2\right\rangle\to H^0(C,\det\sE))
\end{equation} or, equivalently,
\begin{equation}
\label{aggiuntazero2}
\omega \in \Ima(H^0(C,\sL)\otimes W\to H^0(C,\det\sE)).
\end{equation} The first map is given by the wedge product, the second one by the fact that $\det\sE\cong\sL\otimes\sF$. Note that if $H^0(C,\sL)=0$ this condition is equivalent to $\omega=0$.
\begin{rmk}
\label{casobase}
If $\sL=\sO_C$, then $\det\sE\cong\sF$ and we can (and will) see $\omega$ as an element of $H^0(C,\sF)$. Condition (\ref{aggiuntazero2}) is easily written as $\omega\in W$.
\end{rmk}

From the natural map
\begin{equation}
\sF^\vee\otimes\sL\to\sF^\vee\otimes\sL(D_W)
\end{equation} we have a homomorphism
\begin{equation}
H^1(C,\sF^\vee\otimes\sL)\stackrel{\rho}{\rightarrow} H^1(C,\sF^\vee\otimes\sL(D_W));
\end{equation} in analogy with Theorem \ref{aggiunta0} we call $\xi_{D_W}:=\rho(\xi)$.

In the case where there are no singularities in the base locus, Theorem \ref{aggiunta0} is formulated in the following easier way:
\begin{cor}
\label{aggiunta}
Let $C$ be an irreducible Gorenstein curve with $n$ nodes. Let $\sF$, $\sL$ be invertible sheaves on $C$. Consider $\xi\in H^1(C,\mathcal{F}^\vee\otimes\sL)$ associated to the extension (\ref{sequenza}). Define $W=\left\langle \eta_1,\eta_{2}\right\rangle\subset\ker({\partial_\xi})\subset H^0(C,\mathcal{F})$ and $\omega$ as above. Assume also that the base locus of $W$ does not contain singularities.
We have that $\omega \in \Ima(H^0(C,\sL)\otimes W\to H^0(C,\det\sE))$ if and only if $\xi_{D_W}=0$.
\end{cor}

\begin{rmk}
If $D_W=0$, then under the same hypothesis we have $\omega \in \Ima(H^0(C,\sL)\otimes W\to H^0(C,\det\sE))$ if and only if $\xi=0$.
\end{rmk}

\section{Torelli-type theorem}
\subsection{Main Theorem}
We want to study the extensions of the dualizing sheaf $\omega_C$ of a Gorenstein curve $C$. We need a version of the Castelnuovo theorem in the case of Gorenstein curves. 
\begin{thm}
\label{castelnuovo1}
Let $C$ be an irreducible Gorenstein curve with $n$ singular points, $|D|$ a base point free linear system of dimension $r\geq3$ and assume that the map 
\begin{equation}
\phi_D\colon C\to \mathbb{P}^r
\end{equation} is birational onto the image.
Then the natural map
\begin{equation}
\text{Sym}^lH^0(C,\sO(D))\otimes H^0(C,\omega_C)\to H^0(C,\omega_C(lD))
\end{equation} is surjective for $l\geq0$.
\end{thm}
Since $|D|$ is base point free then its general member does not contain singular points and $D$ is a Cartier divisor. Hence the proof is very similar to the one in the case of smooth curves; see \cite[Page 151]{ACGH}.

We can prove the following theorem. The version for smooth curves can be found in \cite{Ce}.
\begin{thm}
\label{teorema1}
Let $C$ be an irreducible Gorenstein curve as above, and let $\sL$ be a locally free sheaf of rank one on $C$. Consider the extension
\begin{equation}
0\to\sO_C\to\sE\to\sL\to0
\label{estensione}
\end{equation} given by an element $\xi\in \text{{\rm Ext}}^1(\sL,\sO_C)\cong H^1(C,\sL^\vee)$. Call $F$ the fixed part of the linear system associated to $\sL$ and assume that $F$ does not contain singularities; call $M$ its mobile part. Assume that the map $\phi_M$ given by $M$ is of degree one and that $l:=\dim |M|\geq3$. If the cohomology map
\begin{equation}
H^0(C,\sE)\to H^0(C,\sL)
\label{partial}
\end{equation} is surjective, then $\xi_F=0$.
\end{thm}
\begin{proof}
Take $P_1,\ldots,P_{l-1}$ points in general position in $\phi_M(C)$ and call $D:=\sum P_i$. The hyperplanes passing through the points $P_i$ form a pencil and we take generators $\eta_1$ and $\eta_2$ of $H^0(C,\sL(-D-F))$. 

We have the following diagram
\begin{equation}
\xymatrix { 
&&0\ar[d]&0\ar[d]\\
0 \ar[r] & \sO_C \ar[r]\ar@{=}[d] &\hat{\mathcal{E}} \ar[r] \ar[d]& \mathcal{L}(-D-F) \ar[d]\ar[r]&0\\
0 \ar[r] &\sO_C\ar[r] &\mathcal{E} \ar[r]\ar[d] & \sL \ar[r]\ar[d]&0 \\
&&\mathcal{L}\otimes\mathcal{O}_{D+F}\ar@{=}[r]\ar[d]&\mathcal{L}\otimes\mathcal{O}_{D+F}\ar[d]\\
&&0&0.
}
\end{equation}
The top row is an extension associated to an element $\hat{\xi}\in H^1(C,\sL^\vee(D+F))$. The sections $\eta_1$ and $\eta_2$ can be lifted to $H^0(C,\hat{\sE})$, furthermore the space $W:=\left\langle \eta_1,\eta_2\right\rangle$ coincides with the whole space $H^0(C,\sL(-D-F))$. Hence Remark \ref{casobase} and Theorem \ref{aggiunta} can be used to deduce that $\hat{\xi}=0$, since $W$ is base point free by the General Position Theorem; see \cite[Page 109]{ACGH}. 

Thus $\xi$ is in the kernel of the map 
\begin{equation}
\label{kernel}
H^1(C,\sL^\vee)\to H^1(C,\sL^\vee(D+F))
\end{equation} for every $D$ as above. We want to show that this implies that $	\xi$ is in the kernel of 
\begin{equation}
H^1(C,\sL^\vee)\to H^1(C,\sL^\vee(F)),
\end{equation} that is $\xi_F=0$.

Since (\ref{kernel}) is surjective, using Serre duality we have the injective map
\begin{equation}
H^0(C,\omega_C\otimes\sL(-D-F))\to H^0(C,\omega_C\otimes\sL),
\end{equation} which factors as
\begin{equation}
H^0(C,\omega_C\otimes\sL(-D-F))\stackrel{\Phi_D}{\rightarrow} H^0(C,\omega_C\otimes\sL(-F))\to H^0(C,\omega_C\otimes\sL).
\end{equation} We consider now $\xi_F$ as an element of the dual of $H^0(C,\omega_C\otimes\sL(-F))$; (\ref{kernel}) implies that $\Ima \Phi_D\subset\ker\xi_F$.

Define
\begin{equation}
K:=\left\langle \bigcup_{\substack{D\in \text{Div}(C)\\\deg D=l-1}}\Ima\Phi_D\right\rangle=\left\langle \bigcup_{\substack{D\in \text{Div}(C)\\\deg D=l-1}}H^0(C,\omega_C\otimes\sL(-D-F))\right\rangle.
\end{equation} We will show that $K=H^0(C,\omega_C\otimes\sL(-F))$. To do this consider 
\begin{equation}
\begin{split}
K':&=\left\langle \bigcup_{\substack{D\in \text{Div}(C)\\\deg D=l-1}}H^0(C,\omega_C)\otimes H^0(C,\sL(-D-F))\right\rangle=\\&=H^0(C,\omega_C)\otimes\left\langle \bigcup_{\substack{D\in \text{Div}(C)\\\deg D=l-1}}H^0(C,\sL(-D-F))\right\rangle.
\end{split}
\end{equation} Since $\left\langle \bigcup_{\substack{D\in \text{Div}(C)\\\deg D=l-1}}H^0(C,\sL(-D-F))\right\rangle $ is equal to $H^0(C,\sL(-F))$ we have
\begin{equation}
K'= H^0(C,\omega_C)\otimes H^0(C,\sL(-F)).
\end{equation} Now it is easy to see that the natural morphism $K'\to K$ fits in the following diagram
\begin{equation}
\xymatrix { 
K\ar@{^{(}->}[r]&H^0(C,\omega_C\otimes\sL(-F))\\
K'\ar[u]\ar@{->>}[ur]&
}
\end{equation} The diagonal arrow is surjective by Theorem \ref{castelnuovo1} and we deduce that $K=H^0(C,\omega_C\otimes\sL(-F))$, as wanted.
\end{proof}

We have proved the main theorem of the introduction.

\subsection{Extension classes of $\omega_C$: the Gorenstein case}
Recall that an integral Gorenstein curve $C$ of arithmetic genus $p_a\geq2$ is \emph{hyperelliptic} if there exist a finite morphism $C\to\mathbb{P}^1$ of degree $2$. Next theorem is useful to check injectivity for the differential of the period map in the Gorenstein case.

\begin{thm}
\label{torelli1}
Let $C$ be an irreducible Gorenstein curve of genus $2$ or non-hyperelliptic of genus $>3$ and let $\xi\in H^1(C,\omega_C^\vee)$ such that $\partial_{\xi}=0$. Then $\xi=0$. 
\end{thm}
\begin{proof}
We start with the case $p_a=2$. Take an element $\xi\in H^1(C,\omega_C^\vee)$. Since $h^0(C,\omega_C)=2$ we have that $W=\left\langle \eta_1,\eta_2\right\rangle=H^0(C,\omega_C)$. By Remark \ref{casobase} and Theorem \ref{aggiunta} we have immediately that $\xi_{D_W}=0$. It is well known, see for example \cite[Theorem 1.6]{H2}, that $H^0(C,\omega_C)$ generates $\omega_C$, hence $D_W=0$ and we are done.
If $p_a>3$ we can apply the previous theorem with $\sL=\omega_C$. Also in this case the base locus is zero, hence $\xi=0$.
\end{proof}

We want now to analyze the remaining case $p_a=3$. We will need the following 
\begin{thm}
\label{noether}
If $C$ is a non-hyperelliptic Gorenstein curve, then the homomorphism
\begin{equation}
\text{Sym}^lH^0(C,\omega_C)\to H^0(C,\omega_C^l)
\end{equation} is surjective for $l\geq1$.
\end{thm}
\begin{proof}
This is the Max Noether theorem for Gorenstein curves. It follows from Theorem \ref{castelnuovo1}. See also \cite{M}.
\end{proof}
\begin{thm}
\label{noether1}
Let $C$ be a non-hyperelliptic Gorenstein curve and $\xi\in H^1(C,\omega_C^\vee)$. Consider the morphism
\begin{equation}
\Phi_P\colon H^1(C,\omega_C^\vee)\to H^1(C,\omega_C^\vee(P))
\end{equation} associated to a smooth point $P\in C$. If $\xi_P:=\Phi_P(\xi)=0$ for every $P$, then $\xi=0$.
\end{thm}
\begin{proof}
The proof follows closely the proof of Theorem \ref{teorema1}, with the only exception that it uses the Max Noether theorem instead of Theorem \ref{castelnuovo1}.
\end{proof}
We can now prove the following
\begin{thm}
\label{torelli2}
Let $C$ be an irreducible Gorenstein curve of genus $2$ or non-hyperelliptic of genus $\geq3$, $\xi\in H^1(C,\omega_C^\vee)$ such that $\partial_{\xi}=0$. Then $\xi=0$. 
\end{thm}
\begin{proof}
Everything except the case $p_a=3$ is proved in Theorem \ref{torelli1}. So assume that $p_a=3$ and take a point $P\in C$. By \cite[Proposition 1.5]{H2} we have 
\begin{equation}
h^0(C,\omega_C(-P))=h^0(C,\omega_C)-1=2.
\end{equation}
Take $\xi\in H^1(C,\omega_C^\vee)$ and consider the extension associated to $\xi_P$, that is the first row of the diagram
\begin{equation}
\xymatrix { 
&&0\ar[d]&0\ar[d]\\
0 \ar[r] & \sO_C \ar[r]\ar@{=}[d] &\hat{\mathcal{E}} \ar[r] \ar[d]& \omega_C(-P) \ar[d]\ar[r]&0\\
0 \ar[r] &\sO_C\ar[r] &\mathcal{E} \ar[r]\ar[d] & \omega_C \ar[r]\ar[d]&0 \\
&&\omega_C\otimes\mathcal{O}_{P}\ar@{=}[r]\ar[d]&\omega_C\otimes\mathcal{O}_{P}\ar[d]\\
&&0&0.
}
\end{equation} By the fact that $\omega_C(-P)$ has only two linearly independent global sections it follows that if we construct the adjoint $\omega$ starting from a base $\eta_1,\eta_2$ of $H^0(C,\omega_C(-P))$, then the adjoint is forced to be a linear combination of $\eta_1,\eta_2$. Hence Theorem \ref{aggiunta} can be applied and $\xi_P=0$ for every $P\in C$ (as in Theorem \ref{torelli1} we use the fact that there are no base points). By Theorem \ref{noether1}, $\xi=0$ and we are done.
\end{proof}

\subsection{Infinitesimal deformations of Gorenstein curves} The above analysis of the extensions of the dualizing sheaf $\omega_C$ of a Gorenstein curve $C$ gives information also on the infinitesimal deformations of $C$. Take $\xi\in\text{Ext}^1(\Omega_C^1,\sO_C)$ in the kernel of (\ref{localtoglobal1}). By (\ref{isomorfismo}), $\xi$ can be seen equivalently as an element of $H^1(C,\mathcal{H}\textit{om}(\Omega_C^1,\sO_C))$ or of $\text{Ext}^1(\hat{\omega},\sO_C)$. Now apply the functor $\text{Hom}(-,\sO_C)$ to (\ref{seq2}) to obtain
{\footnotesize
\begin{equation}
0\to \text{Hom}(\omega_C,\sO_C)\to \text{Hom}(\hat{\omega},\sO_C)\to\text{Ext}^1(N,\sO_C)\to\text{Ext}^1(\omega_C,\sO_C)\to\text{Ext}^1(\hat{\omega},\sO_C)\to0.
\end{equation}}
Therefore we can find a lifting $\tilde{\xi}\in\text{Ext}^1(\omega_C,\sO_C) $ of $\xi$. Of course if $\tilde{\xi}$ is zero, then also $\xi$ is zero. We have proved the following
\begin{thm}
\label{deformazionigorenstein} Let $C$ be an irreducible Gorenstein curve and let $\xi\in\text{{\rm Ext}}^1(\Omega^1_C,\sO_C)$ be an infinitesimal deformation of $C$. Assume that $\xi\in \ker (\mu\colon\text{{\rm Ext}}^1(\Omega^1_C,\sO_C)\to R)$, see Proposition \ref{localtoglobal1}. Then $\xi$ can be lifted to an element $\tilde{\xi}$ of $\text{{\rm Ext}}^1(\omega_C,\sO_C)$.  Moreover if $C$ is of genus $2$ or non-hyperelliptic of genus $\geq3$ and 
$\partial_{\tilde{\xi}}=0$ then $\xi=0$.
\end{thm}
\begin{rmk}
In general it seems not a trivial problem to use the condition $\partial_{\tilde{\xi}}=0$ to obtain $\partial_\xi=0$.
\end{rmk}

\section{Reducible nodal curves}

Unlike the irreducible case, for reducible Gorenstein curves there is no hope to have, in general, an infinitesimal Torelli-type theorem. Consider $C$ a nodal curve with two components meeting transversely in $n$ nodes. Note that all the results of Section \ref{sezione1} still hold in this case.

Sequence (\ref{localtoglobal2}) and isomorphism (\ref{isonorm}) give the short exact sequence
\begin{equation}
\label{riducibile}
0\to H^1(C_1,\Theta_{C_1}(-\sum P_i))\oplus H^1(C_2,\Theta_{C_2}(-\sum P_i))\to \text{Ext}^1(\Omega_C^1,\sO_C)\to \bigoplus_{i=1}^n\mathbb{C}_{P_i}\to0
\end{equation} where $C_1$ and $C_2$ are the disjoint components of the normalization of $C$. From now on call $D:=\sum P_i$ the divisor on $C_i$, $i=1,2$, induced by the intersection points.

As in the irreducible case we study the infinitesimal deformations coming from the kernel of this sequence, that is the deformations $\xi\in \text{Ext}^1(\Omega_C^1,\sO_C)$ which can be written as $\xi=\xi_1\oplus\xi_2$, with $\xi_i\in H^1(C_i,\Theta_{C_i}(-D))$ associated to an exact sequence 
\begin{equation}
\label{seqi}
0\to \sO_{C_i}\to \sE_i\to \omega_{C_i}(D)\to 0.
\end{equation}

Note that the curves $C_i$ are smooth, hence $\Omega_{C_i}^1$ is the dualizing sheaf $\omega_{C_i}$ of $C_i$.

%Consider now $\xi_i$. It is associated to an exact sequence 
%\begin{equation}
%0\to \sO_{C_i}\to \sE_i\to \omega_{C_i}(P)\to 0.
%\end{equation}

\begin{defn}
We say that $\xi$ satisfies the \emph{split liftability condition} if the deformations $\xi_i$, $i=1,2$, lift all the corresponding global sections in (\ref{seqi}), that is $\partial_{\xi_1}=\partial_{\xi_2}=0$. 
\end{defn}

In the final subsections we present explicit cases where even if $\xi$ satisfies the split liftability condition we have that $\xi\neq0$. We want to stress that in the last examples presented in this paper we have also that $\partial_\xi=0$. In other words the notation $\xi=\xi_1\oplus\xi_2$ does not force the implication $\partial_{\xi_i}=0\Rightarrow\partial_\xi=0$. 

To see which deformations can be used to provide such examples note that the image of $\xi_i$ in $H^1(C_i,\Theta_{C_i})$ gives an extension of $\omega_{C_i}$ which fits into the following diagram
\begin{equation}
\label{diagramma2}
\xymatrix { 
&&0\ar[d]&0\ar[d]\\
0 \ar[r] & \sO_{C_i} \ar[r]\ar@{=}[d] &\hat{\mathcal{E}_i} \ar[r] \ar[d]& \omega_{C_i} \ar[d]\ar[r]&0\\
0 \ar[r] &\sO_{C_i}\ar[r] &\mathcal{E}_i \ar[r]\ar[d] & \omega_{C_i}(D) \ar[r]\ar[d]&0 \\
&&\omega_{C_i}\otimes\mathcal{O}_{D}\ar@{=}[r]\ar[d]&\omega_{C_i}\otimes\mathcal{O}_{D}\ar[d]\\
&&0&0.
}
\end{equation} 

Now assume that the map $H^0(C_i,\sE_i) \to H^0(C_i, \omega_{C_i}(D) )$ is surjective, that is $\partial_{\xi_i}=0$, then it is easy to see that $H^0(C_i,\hat{\sE_i}) \to H^0(C_i, \omega_{C_i})$ is also surjective. If, for example, $C_i$ is of genus 2 or non-hyperelliptic of genus $\geq3$, then we deduce that the first row of (\ref{diagramma2}) splits. Hence $\xi_i$ is in the kernel of 
\begin{equation}
H^1(C_i,\Theta_{C_i}(-D))\to H^1(C_i,\Theta_{C_i}).
\end{equation} By the exact sequence 
\begin{equation}
0\to \Theta_{C_i}(-D)\to \Theta_{C_i}\to \Theta_{C_i}|_{D}\to0
\end{equation} we know that this kernel is not zero.

Thus a deformation $\xi_i$ with $\partial_{\xi_i}=0$ is not necessarily the trivial deformation even in the usual case with $C_i$ of genus 2 or non-hyperelliptic of genus $\geq3$. 

\subsection{Split liftability and nonzero deformations}
Consider $C=R+\hat{C}$ with $R$ a smooth rational curve and $\hat{C}$ an arbitrary smooth curve meeting transversely in $n$ points. We want to study how the situation varies with $n$.

If $n=1$, sequence (\ref{riducibile}) is
\begin{equation}
0\to H^1(\hat{C},\Theta_{\hat{C}}(-P))\oplus H^1(R,\Theta_{R}(-P))\to \text{Ext}^1(\Omega_C^1,\sO_C)\to \mathbb{C}_P\to0.
\end{equation} In this case $\omega_{\hat{C}}(P)$ has the base point $P$ hence an infinitesimal deformation of $\hat{C}$ that lifts everything may be different from zero and supported on $P$; see Theorem \ref{teorema1}.

If $n=2$, the situation is more complicated. Sequence (\ref{riducibile}) is 
\begin{equation}
0\to H^1(\hat{C},\Theta_{\hat{C}}(-P-Q))\to \text{Ext}^1(\Omega_C^1,\sO_C)\to \mathbb{C}_P\oplus\mC_Q\to0
\end{equation} because $H^1(R,\Theta_{R}(-P-Q))=g(R)=0$ since $R$ is rational.
This means that our deformation $\xi\in \text{Ext}^1(\Omega_C^1,\sO_C)$ comes from an element $\xi_{\hat{C}}\in H^1(\hat{C},\Theta_{\hat{C}}(-P-Q))$ (recall that we are studying the deformations in the kernel of (\ref{riducibile})). Note that the sheaf $\omega_{\hat{C}}(P+Q)$ is base point free. Now if $\hat{C}$ is of genus 2 or non-hyperelliptic of genus $\geq3$ the map induced by the linear system $\omega_{\hat{C}}(P+Q)$ is of rank one since the canonical map of $\hat{C}$ is already of rank one. Hence in this case Theorem \ref{teorema1} applies and the lifting hypothesis $\partial_{\xi_{\hat{C}}}=0$ implies $\xi_{\hat{C}}=0$. If $\hat{C}$ is hyperelliptic, the same conclusions holds if $P+Q$ is not the $g^1_2$ linear system. In fact the map given by $\omega_{\hat{C}}(P+Q)$ has degree $\leq2 $ since the canonical map has degree $2$ and we prove that it has degree $2$ only if $P+Q$ is the $g^1_2$. In fact assume that the degree is $2$, then the image of $\hat{C}$ in $\mP(H^0(\omega_{\hat{C}}(P+Q))^\vee)=\mP^g$ is a (nondegenerate) curve of degree $g$. Hence it is projectively isomorphic to the rational normal curve; see \cite[Page 179]{GH}. The morphism factors as
\begin{equation}
\hat{C}\stackrel{2:1}{\rightarrow}\mP^1\to\mP^g
\end{equation} and $P+Q$ is exactly the $g^1_2$. 

If $n=3$, sequence (\ref{riducibile}) is
\begin{equation}
0\to H^1(\hat{C},\Theta_{\hat{C}}(-\sum_{i=1}^3 P_i))\to \text{Ext}^1(\Omega_C^1,\sO_C)\to \bigoplus_{i=1}^3\mathbb{C}_{P_i}\to0
\end{equation} since $H^1(R,\Theta_{R}(-\sum P_i))=0$ by duality. Hence as in the previous case we deal only with $\xi_{\hat{C}}$. Assume that $\partial_{\xi_{\hat{C}}}=0$. Since $\omega_{\hat{C}}(\sum P_i)$ is very ample, by Theorem \ref{teorema1} we have that $\xi_{\hat{C}}=0$ and we deduce that $\xi=0$. This means that in this case the split liftability condition implies that $\xi=0$.

On the other hand assume that $n>3$. Sequence (\ref{riducibile}) is
\begin{equation}
0\to H^1(\hat{C},\Theta_{\hat{C}}(-\sum P_i))\oplus H^1(R,\Theta_{R}(-\sum P_i)) \to \text{Ext}^1(\Omega_C^1,\sO_C)\to \bigoplus_{i=1}^n\mathbb{C}_{P_i}\to0.
\end{equation} Now $H^1(R,\Theta_{R}(-D))\neq 0$, so take a nonzero element $\xi_R$. It corresponds to a sequence
\begin{equation}
0\to\sO_R\to \sE_R\to \omega_R(\sum P_i)\to0.
\end{equation} By the fact that $H^1(R,\sO_R)=0$, we immediately deduce that $\partial_{\xi_R}=0$, whereas $\xi_R\neq0$. Hence taking $\xi=0\oplus \xi_R$ we have that in this case the split liftability condition does not imply that $\xi=0$.

\subsection{Liftability of global sections and nonzero deformations}

Take a plane curve $C=C_1+C_2$ where $C_1$ is a smooth quadric and $C_2$ is a smooth cubic. As before we are interested in the deformations $\xi\in\text{Ext}^1(\hat{\omega},\sO_C)=H^1(C,\hat{\omega}^\vee)$. To show that infinitesimal Torelli theorem does not hold in this case we have to show that the map 
\begin{equation}
H^1(C,\hat{\omega}^\vee)\to H^0(C,\hat{\omega})^\vee\otimes H^1(C,\sO_C)
\end{equation} is not injective. This map is exactly (\ref{problema}) in the case we are studying. Dualizing, we have to show that
\begin{equation}
\label{mappasur}
H^0(C,\omega_C)\otimes H^0(C,\hat{\omega})\to H^0(C,\hat{\omega}\otimes \omega_C)
\end{equation} is not surjective.

By \cite[Sequence 2.13 page 91]{ACG} we have that $\hat{\omega}\cong\nu_*(\omega_{\tilde{C}})$, so we can easily compute the dimensions of the vector spaces appearing in (\ref{mappasur}). We obtain that $h^0(C,\omega_C)=6$ and $h^0(C,\hat{\omega})=1$. On the other hand $\nu^*(\omega_C)=\omega_{\tilde{C}}(D)$ where $D$ is the divisor induced by the intersection points on the normalization $\tilde{C}$; see \cite[Page 101]{ACG}. Hence we have by the projection formula
\begin{equation}
h^0(\hat{\omega}\otimes \omega_C)=h^0(\nu_*(\omega_{\tilde{C}})\otimes\omega_C)=h^0(\nu_*(\omega_{\tilde{C}}\otimes\omega_{\tilde{C}}(D)))=9.
\end{equation}
Hence (\ref{mappasur}) is not surjective.

Another example is given by $C=C_1+C_2$ with $C_1$ a smooth quartic and $C_2$ a line. We consider again the map
\begin{equation}
H^0(C,\omega_C)\otimes H^0(C,\hat{\omega})\to H^0(C,\hat{\omega}\otimes \omega_C).
\end{equation} This map fits into the diagram
\begin{equation}
\xymatrix { &0\ar[d]\\
H^0(C,\omega_C)\otimes H^0(C,\hat{\omega})\ar[d]\ar[r]&H^0(C,\hat{\omega}\otimes \omega_C)\ar[d]\\
\text{Sym}^2H^0(C,\omega_C)\ar[r]&H^0(C,\omega_C^{\otimes2}).
}
\end{equation} The second row is surjective by \cite[Theorem 1]{BB} and it is easy to compute that its kernel $K$ has dimension $6$. Call $V$ the image of $H^0(C,\omega_C)\otimes H^0(C,\hat{\omega})$ in $\text{Sym}^2H^0(C,\omega_C)$. $V$ has dimension $15$ whereas $h^0(C,\hat{\omega}\otimes \omega_C)=11$, hence it will be enough to prove that $K':=\ker(V\to H^0(C,\hat{\omega}\otimes \omega_C))$ has dimension $\geq5$.

We have that $h^0(C,\hat{\omega})=3$ and $h^0(C,\omega_C)=6$, hence take a basis $l_1,l_2,l_3$ of $H^0(C,\hat{\omega})$ and complete it to a basis $l_1,l_2,l_3,b_1,b_2,b_3$ of $H^0(C,\omega_C)$. The sheaf $\omega_C$ is by adjunction $\sO_C(2)$ and its global sections can be seen as quadrics restricted to $C$. Recall by Section \ref{sezione1} that $\hat{\omega}$ is the subsheaf of $\omega_C$ consisting of the sections vanishing on the nodes of $C$, therefore $l_1,l_2,l_3$ are quadrics vanishing on the four nodes of $C$, and hence on the line $C_2$.

Now the elements of $K$ are quadrics vanishing on the canonical image of $C\subset\mP^5$ and an element of $K$ is in $K'$ if it is zero when restricted on the plane $(l_1=l_2=l_3=0)\subset\mP^5$.
The image of $C_2$ in $\mP^5$ is a smooth rational curve contained in this plane and it is easy to see by the linear independence of $l_1,l_2,l_3,b_1,b_2,b_3$ that the dimension of $K'$ is at most one less than the dimension of $K$, hence we are done.

\end{document}